\documentclass[12pt]{article}
\usepackage{graphicx}
\usepackage{verbatim}
\usepackage{textcomp}
\usepackage{amssymb}
\usepackage{cite}
\usepackage{amsmath}
\usepackage{latexsym}
\usepackage{amscd}
\usepackage{amsthm}
\usepackage{mathrsfs}
\usepackage{xypic}
\usepackage{bm}
\usepackage{url}
\usepackage{hyperref}
\usepackage[latin1]{inputenc}
\usepackage{enumitem}
\allowdisplaybreaks[4]
\makeatletter \@addtoreset{equation}{section}

\newtheorem{thm}{Theorem}[section]
\newtheorem{corr}[thm]{Corollary}
\newtheorem{lem}[thm]{Lemma}

\theoremstyle{definition}

\theoremstyle{remark}
\newtheorem{rem}{Remark}[section]
\numberwithin{equation}{section}
\begin{document}
\begin{sloppypar}
\title{{Hopf's lemma for parabolic equations involving a generalized tempered fractional $p$-Laplacian}
\footnotetext{{Keywords: Hopf's lemma, a generalized tempered fractional $p$-Laplacian, parabolic fractional Laplacian, maximum principle.}}}

\author{\small  \it Linlin Fan, Linfen Cao and Peibiao Zhao\footnote{Corresponding author: Peibiao Zhao}
 }

\date{}
\maketitle

\begin{abstract}
In this paper, we study a nonlinear system involving a generalized tempered fractional $p$-Laplacian in $B_{1}(0)$:
\begin{equation*}
\left\{
\begin{array}{ll}
\partial_tu(x,t)+(-\Delta-\lambda_{f})_{p}^{s}u(x,t)=g(t,u(x,t)),
&(x,t)\in B_{1}(0)\times[0,+\infty),\\
u(x)=0,&(x,t)\in B_{1}^{c}(0)\times[0,+\infty),
\end{array}
\right.
\end{equation*}
where $0<s<1$, $p>2,\ n\geq2$. We establish Hopf's lemma for parabolic equations involving a generalized tempered fractional $p$-Laplacian. Hopf's lemma will become powerful tools in obtaining qualitative properties of solutions for nonlocal parabolic equations.
\end{abstract}
\section{Introduction}
\ \ \ \ \ \ Our main concern in this paper is to study the following nonlinear system involving the generalized tempered fractional $p$-Laplacian in $B_{1}(0)$:
\begin{equation}\label{eq:1}
\left\{
\begin{array}{ll}
\partial_tu(x,t)+(-\Delta-\lambda_{f})_{p}^{s}u(x,t)=g(t,u(x,t)),
&(x,t)\in B_{1}(0)\times[0,+\infty),\\
u(x)=0,&(x,t)\in B_{1}^{c}(0)\times[0,+\infty),
\end{array}
\right.
\end{equation}
where $0<s<1$, $p>2,\ n\geq2$.
Here, $(-\Delta-\lambda_{f})_{p}^{s}$ is a kind of tempered fractional
$p$-Laplacian defined by
\begin{equation}\label{eq:10}
(-\Delta-\lambda_{f})_{p}^{s}u(x)=C_{n,sp}\textit{PV}\int_{\mathbb{R}^{n}}\frac{|u(x)-u(y)|^{p-2}(u(x)-u(y))}{e^{\lambda f(|x-y|)}|x-y|^{n+sp}}dy,
\end{equation}
where $C_{n,sp}$ is a normalization positive constant depending only on $n$, $s$, $p$. $\lambda$ is a sufficiently small positive constant and $f$ is a nondecreasing map with respect to $|x-y|$.

For simplicity, let $G(t)=|t|^{p-2}t$ and then$$(-\Delta-\lambda_{f})_{p}^{s}u(x)=C_{n,sp}\textit{PV}\int_{\mathbb{R}^{n}}\frac{G(u(x)-u(y))}{e^{\lambda f(|x-y|)}|x-y|^{n+sp}}dy.$$ Obviously, $G(t)$ is a strictly increasing function since $G'(t)=(p-1)|t|^{p-2}\geq0$.

Multiplying fractional operators by an exponential factor leads to a tempered fractional operator. Tempered fractional derivatives, in fractional diffusion equation, govern the limits of random walk models with an exponentially tempered power-law jump distribution. Indeed, it describes the transition between standard and anomalous diffusion(see e.g. \cite{SWB2020}). More applications for the tempered fractional derivatives and tempered differential equations can be found, for instance, in poroelasticity, finance, ground water hydrology, and geophysical flows. Tempered fractional calculus can be recognized as the generalization of fractional calculus.

As a special case, when $p=2$ and $f$ is an identity map, the generalized nonlinear tempered fractional $p$-Laplacian becomes the tempered fractional Laplacian $(\Delta+\lambda)^{\frac{\beta}{2}}$,  physically introduced and mathematically defined in \cite{DLTZ2018} as
$$(\Delta+\lambda)^{\frac{\beta}{2}}u(x)=-C_{n,\beta,\lambda}\textit{PV}\int_{\mathbb{R}^{n}}\frac{u(x)-u(y)}{e^{\lambda |x-y|}|x-y|^{n+\beta}}dy,$$
where $\lambda$ is a sufficiently small positive constant, $C_{n,\beta,\lambda}=\frac{\Gamma(\frac{n}{2})}{2\pi^{\frac{n}{2}}|\Gamma(-\beta)|}(\beta\in(0,2))$ and $\Gamma(t)=\int_{0}^{\infty}s^{t-1}e^{-s}ds$ in the Gamma function(see e.g.\cite{DLTZ2018,ZHAW2021}). In 1996, Bertoin \cite{JBertoin1996} interpreted the fractional Laplacian as an infinitesimal generator for a stable L\'{e}vy diffusion process. The scaling limit of L\'{e}vy flight is the $\beta-$stable L\'{e}vy process, generated by the fractional Laplacian $(\Delta)^{\frac{\beta}{2}}$. In order to make the L\'{e}vy flight a more suitable physical model, the concept of the tempered L\'{e}vy flight was introduced. The scaling limit of the tempered L\'{e}vy flight is called the tempered L\'{e}vy process, which is generated by the tempered fractional Laplacian $(\Delta+\lambda)^{\frac{\beta}{2}}$.

%for several motivations concerning the tempered fractional p-Laplacian operator, we refer the reader to.
In \cite{ZHAW2021}, the authors first introduced a new kind of tempered
fractional $p$-Laplacian $(-\Delta-\lambda_{f})_{p}^{s}$ based on tempered fractional Laplacian $(\Delta+\lambda)^{\frac{\beta}{2}}$, which was originally defined in \cite{DLTZ2018} by Deng et.al. In \cite{ZHAW2021}, the authors investigated radial symmetry and monotonicity of positive solutions to a logarithmic Choquard equation involving a generalized nonlinear tempered fractional $p$-Laplacian operator by applying the direct method of moving planes
\begin{equation}
\left\{
\begin{array}{ll}
(-\Delta-\lambda_{f})_{p}^{s}m(x)
\\ \quad=C_{n,t}(|x|^{2t-n}\ast[\ln(m(x)+1)]^{q})[\ln(m(x)+1)]^{q-1},\ &x\in\mathbb{R}^{n},\\
m(x)>0,\ &x\in\mathbb{R}^{n}.
\end{array}
\right.
\end{equation}
Then they discussed the decay of solutions at infinity
and narrow region principle, which play a key role in obtaining the main result by the process of moving planes.

As another special case, $(-\Delta-\lambda_{f})_{p}^{s}$ turns into the fractional $p$-Laplaccian $(-\Delta)_{p}^{s}$ when $\lambda=0$, we refer the reader to \cite{chen2018maximum,cao2018radial,chen2019symmetry,wang2020hopf,CF2021,CF2022},
when $p\neq2$, the fractional $p$-Laplacian is nonlinear and nonuniform elliptic. Moreover, neither any extension method nor the integral equations method has been found for the fractional $p$-Laplacian, as far as we know. Recently, in \cite{chen2018maximum}, Chen and Li introduced a key boundary estimate lemma, which plays the role of Hopf's lemma. Together with the direct method of moving planes\cite{Chen2017a}, symmetry and monotonicity result for equation $(-\Delta)_{p}^{s}u(x)=g(u(x)),x\in\mathbb{R}^{n}$ is obtained. Furthermore, $(-\Delta)_{p}^{s}$ takes the form of the well-known fractional Laplacian $(-\Delta)^{s}$ when $p=2$, we refer the reader to \cite{wang2015radial,P1979C, B2002M, M2013P, E2009P, cao2021radial}.

In order to guarantee the integrability of \eqref{eq:10}, we require that$$u\in L_{sp}(\mathbb{R}^{n})\cap C_{loc}^{1,1}(\mathbb{R}^{n}),$$ with $$L_{sp}=\{u\in L_{loc}^{p-1}|\int_{\mathbb{R}^{n}}\frac{|1+u(x)|^{p-1}}{(1+|x|)^{n+sp}}dx<\infty\}.$$

In recent years, the nonlocal operators arise from many fields, such as game theory, finance, L\'{e}vy processes, and optimization. The nonlocal nature of these operators make them difficult to study. To circumvent this, Caffarelli and Silvestre \cite{caffarelli2007extension} introduced the extension method which turns the nonlocal problem involving the fractional Laplacian into a local one in higher dimensions. This method has been applied successfully to study equations involving the fractional Laplacian, and a series of fruitful results have been obtained (see \cite{brandle2013concave,chen2016indefinite} and the references therein). One can also use the integral equations method,  such as the method of moving planes in integral forms and regularity lifting to investigate equations involving the fractional Laplacian by first showing that they are equivalent to the corresponding integral equations \cite{chen2015liouville,chen2003qualitative,chen2006classification}.

Reaction-diffusion equations and systems have been studied very extensively in the past few years as models for many problems arising in applications such as, quasi-geostrophic flow \cite{Caffarelli2010Vasseur}, general shadow and activator-inhibitor system \cite{Ni2011the}, nonlocal diffusions \cite{BPSV2014,BCVE2016}, and the fractional porous medium \cite{dePabloQRV2011}.

Hopf's lemma is a classic result in analysis, dating back to the discovery of the maximum principles for harmonic functions\cite{Hopf1927}, and it has become a fundamental and powerful tool in the study of partial differential equations.

Recently, with the extensive study of fractional Laplacians and fractional $p$-Laplacians, a series of fractional version of Hopf's lemmas have been established. For instance, Birkner, L$\acute{o}$pez-Mimbela and Wakolbinger \cite{BLMW2005} introduced a fractional version of Hopf's lemmas for anti-symmetric functions which can be applied immediately to the method of moving planes to establish qualitative properties, such as symmetry and montonicity of solutions for fractional equations. Li and Chen \cite{LCA2019} proved a Hopf type lemma for anti-symmetric solutions to fractional equations by direct estimations. Jin and Li \cite{JL2019} derived a Hopf type lemma for anti-symmetric solutions to fractional $p$-equations. Chen, Li and Qi \cite{CLQ2020} obtained a Hopf type lemma for positive weak super-solutions of fractional $p$-equations. Li and Zhang \cite{li2021zhangsub} derived a Hopf type lemma for positive classical solutions of fractional $p$-equations with Dirichlet conditions.

However, for parabolic fractional equations, there have been very few such results(see e.g. \cite{CWNH2021Asymptotic}). So far as we know is Jin and Xiong's article \cite{JX2014} in which they established a strong maximum principle and a Hopf type lemma for odd solutions of some linear fractional parabolic equations in a finite time with given initial conditions. Lu and Zhu \cite{LZ2015} used the Hopf type lemma to prove maximum principles of fully nonlinear equations. Chen and Wu \cite{chen2021liouville} applied Hopf's lemmas for antisymmetric functions to obtain monotonicity for entire solutions of parabolic fractional equations in a half space. Wang and Chen \cite{wang2020hopf} first established Hopf's lemmas for parabolic fractional $p$-equations for $p\geq2$. Then they derived an asymptotic Hopf's lemma for anti-symmetric solutions to parabolic fractional Laplacians.

Motivated by the above interesting literature, in this paper we consider the Hopf's lemma of the following equation
\begin{equation}\label{eq25}
\left\{
\begin{array}{ll}
\partial_tu(x,t)+(-\Delta-\lambda_{f})_{p}^{s}u(x,t)=g(t,u(x,t)),
&(x,t)\in B_{1}(0)\times[0,+\infty),\\
u(x,t)>0,&(x,t)\in B_{1}(0)\times[0,+\infty),\\
u(x)=0,&(x,t)\in B_{1}^{c}(0)\times[0,+\infty),
\end{array}
\right.
\end{equation}
where $0<s<1$, $p>2,\ n\geq2$. We assume that the function $g(t,u):\ [0,+\infty)\times[0,+\infty)\rightarrow \mathbb{R}$ satisfies

$(1)$ $g(t,u)$ is $C_{loc}^{\frac{\tau}{2s}}$ for $t$, and Lipschitz continuous in $u$ uniformly with respect to $t$, where $\tau\in(0,1)$ satisfies $\frac{\tau}{2s}\in(0,1)$.

$(2)$ $g(t,u)$ satisfies $g(t,0)=0$ and $g_{u}(t,0)\leq0$.

$(3)$ $g_{u}=\frac{\partial g}{\partial u}$ is continuous near $u=0$.

To illustrate the main results of this paper, we start by presenting the notations that
will be used in what follows. Let
\begin{equation*}
T_{\alpha}:=\{{x=(x_{1},x_{2},\ldots,x_{n})\in \mathbb{R}^{n}\ |x_{1}=\alpha,\ \alpha\in\mathbb{R}}\}
\end{equation*}
be the moving planes perpendicular to $x_{1}$-axis,
\begin{equation*}
\Sigma_{\alpha}=\{x\in\mathbb{R}|x_{1}<\alpha\}
\end{equation*}
be the region to the left of the hyperplane $T_{\alpha}$ in $\mathbb{R}^{n}$. We denote the reflection of $x$ with respect to the hyperplane $T_{\alpha}$ as
\begin{equation*}
x^{\alpha}=(2\alpha-x_{1},x_{2},\ldots,x_{n}).
\end{equation*}
Let $u(x,t)$ be a solution of \eqref{eq25} and $u_{\alpha}(x,t):=u(x^{\alpha},t)$. Thus we define
 \begin{equation*}
w_{\alpha}(x,t)=u_{\alpha}(x,t)-u(x,t).
\end{equation*}

To study monotonicity of solutions, we consider the well-known $w-$limit set of $u$
$$w(u):=\{\varphi|\varphi=\lim u(\cdot,t_{k})\ \mbox{for some}\ t_{k}\rightarrow\infty\},$$
with the limit in $C_{0}(B_{1}(0))$. One can derive that $w(u)$ is a nonempty compact subset of $C_{0}(B_{1}(0))$, for each $\varphi(x)\in w(u)$, denote
$$\psi_{\alpha}(x)=\varphi(x^{\alpha})-\varphi(x)=\varphi_{\alpha}(x)-\varphi(x).$$
Obviously, it is the $w-$limit of $w_{\alpha}(x,t)$.

We establish
\begin{thm}
(Asymptotic Hopf's lemma for a generalized tempered fractional p-Laplacian) Assume that $u(x,t)\in(C_{loc}^{1,1}(B_{1}(0))\cap\mathcal{L}_{sp})\times C^{1}(0,\infty)$ is a positive solution to
\begin{equation}
\partial_t u(x,t)+(-\Delta-\lambda_{f})_{p}^{s}u(x,t)=g(t,u(x,t)),
\end{equation}
where $0<s<1$, $p>2,\ n\geq2$.
%and $\Omega$ is a bounded domain in $B_{1}(0)$ with smooth boundary.

Assume that
\begin{equation}
g(t,0)=0,\ g\ \mbox{is Lispchitz continuous in u uniformly fo t}.
\end{equation}
Then there exists a positive constant $c$, such that for any $t\rightarrow\infty$ and for all $x$ near the boundary of $B_{1}(0)$, we have
\[\varphi(x)\geq c d^{s}(x),\]
where $d(x)=dist(x,\partial B_{1}(0))$. It follows that
\[\frac{\partial \varphi}{\partial\nu^{s}}(x_{0})<0,\]
where $\nu$ is the outward normal of $\partial B_{1}(0)$ at $x_{0}$ and $\frac{\partial \varphi}{\partial\nu^{s}}$ is the normal derivative of fractional order $s$.
\end{thm}

Throughout the paper, C will be positive constants which can be different from line to line and only the relevant dependence is specified.
\section{Some lemmas}
\begin{lem}\label{lem2}
(A maximum principle for a generalized tempered fractional p-Laplacian) Let $\Omega\subset\Sigma_{\alpha}$ be a bounded or unbounded region. Assume that $w_{\alpha}\in L_{sp}\cap C_{loc}^{1,1}(\Omega)$ is lower semi-continuous on $\bar{\Omega}$ and uniformly bounded about $t$,
\begin{equation}\label{eq7}
\left\{
\begin{array}{ll}
\partial_tw_{\alpha}(x,t)+(-\Delta-\lambda_{f})_{p}^{s}u_{\alpha}(x,t)-(-\Delta-\lambda_{f})_{p}^{s}u(x,t)
\\ \quad=c(x,t)w_{\alpha}(x,t),
&(x,t)\in\Omega\times[0,+\infty),\\
w_{\alpha}(x,t)\geq0,
&(x,t)\in(\Sigma_{\alpha}\setminus\Omega)\times[0,+\infty),\\
w_{\alpha}(x^{\alpha},t)=-w_{\alpha}(x,t),&x\in\Sigma_{\alpha}\times[0,+\infty).
\end{array}
\right.
\end{equation}
If $$c(x,t)<0\ \ (x,t)\in\Omega\times[0,+\infty),$$ then
\begin{equation}
w_{\alpha}(x,t)\geq0,\ (x,t)\in\Sigma_{\alpha}\times[0,+\infty).
\end{equation}
 Furthermore, the following strong maximum principle also holds:\\
Either\begin{equation}
w_{\alpha}(x,t)\equiv0,\ \ (x,t)\in\Omega\times[0,+\infty)
\end{equation}
or\begin{equation}\ w_{\alpha}(x,t)>0,\ (x,t)\in\Omega\times[0,+\infty).
\end{equation}
\end{lem}
\begin{proof} Let $x=(x_{1},x')$, denote $$\bar{w}_{\alpha}(x,t)=e^{qt}w_{\alpha}(x,t),$$
where $q>0$ is a constant.

According to the definition of $\bar{w}_{\alpha}(x,t)$ and \eqref{eq7}, $\bar{w}(x,t)$ satisfies
\begin{equation}\label{eq51}
\left\{
\begin{array}{ll}
\partial_t\bar{w}_{\alpha}(x,t)+e^{qt}[(-\Delta-\lambda_{f})_{p}^{s}u_{\alpha}(x,t)-(-\Delta-\lambda_{f})_{p}^{s}u(x,t)]
\\ \quad=(c(x,t)+q)\bar{w}_{\alpha}(x,t),
&(x,t)\in\Omega\times[0,+\infty),\\
\\ \bar{w}_{\alpha}(x,t)\geq0,
&(x,t)\in(\Sigma_{\alpha}\setminus\Omega)\times[0,+\infty).
\end{array}
\right.
\end{equation}

Now we show
\begin{equation}\label{eq11}
\bar{w}_{\alpha}(x,t)\geq\min\{0,\underset{\Omega}\inf\bar{w}_{\alpha}(x,\underline{t})\},\ (x,t)\in\Omega\times[\underline{t},T]
\end{equation}
for any $[\underline{t},T]\subset[0,+\infty)$. Otherwise, there exists $(x_{0},t_{0})\in\Omega\times(\underline{t},T]$ such that
$$\bar{w}_{\alpha}(x_{0},t_{0})=\underset{\Sigma_{\alpha\times(\underline{t},T]}}\inf\bar{w}_{\alpha}(x,t)<\min\{0,\underset{\Omega}\inf\bar{w}_{\alpha}(x,\underline{t})\},$$
then
\begin{equation}\label{eq8}
\partial_t\bar{w}_{\alpha}(x_{0},t_{0})\leq0,
\end{equation}
and
$$\bar{w}_{\alpha}(x_{0},t_{0})-\bar{w}_{\alpha}(y,t_{0})\leq0,\ \ y\in\Sigma_{\alpha},$$
\begin{equation}\label{LS1C2}
w_{\alpha}(x_{0},t_{0})-w_{\alpha}(y,t_{0})=\frac{1}{e^{qt_0}}(\bar{w}_{\alpha}(x_{0},t_{0})-\bar{w}_{\alpha}(y,t_{0}))\leq0,\ \ y\in\Sigma_{\alpha}.
\end{equation}
It follows that
\begin{equation}\label{LS1C3}\begin{split}
0&\geq w_{\alpha}(x_{0},t_{0})-w_{\alpha}(y,t_{0})\\
&=u_{\alpha}(x_{0},t_{0})-u(x_{0},t_{0})-(u_{\alpha}(y,t_{0})-u(y,t_{0}))\\
&=(u_{\alpha}(x_{0},t_{0})-u_{\alpha}(y,t_{0}))-(u(x_{0},t_{0})-u(y,t_{0})).
\end{split}
\end{equation}
Combing above and the fact that $G(t)$ is strictly increasing, we have
\begin{equation}\label{LS1C3}
0\geq G(u_{\alpha}(x_{0},t_{0})-u_{\alpha}(y,t_{0}))-G(u(x_{0},t_{0})-u(y,t_{0})),\ \  y\in\Sigma_{\alpha}.
\end{equation}
Noting that $|x_{0}-y|\leq|x_{0}-y^{\alpha}|$ and $f$ is a nondecreasing map with respect to $|x-y|$, we derive that
\begin{equation}\label{LS1C4}
\frac{1}{e^{\lambda f(|x_{0}-y|)}|x_{0}-y|^{n+sp}}
\geq\frac{1}{e^{\lambda f(|x_{0}-y^{\alpha}|)}|x_{0}-y^{\alpha}|^{n+sp}},\ \  y\in\Sigma_{\alpha}.
\end{equation}

Combining \eqref{LS1C3} and \eqref{LS1C4}, we have the following inequality:
\begin{equation}\label{LS1C5}\begin{split}
&\frac{G(u_{\alpha}(x_{0},t_{0})-u_{\alpha}(y,t_{0}))-G(u(x_{0},t_{0})-u(y,t_{0}))}{e^{\lambda f(|x_{0}-y|)}|x_{0}-y|^{n+sp}}\\
\leq&\frac{G(u_{\alpha}(x_{0},t_{0})-u_{\alpha}(y,t_{0}))-G(u(x_{0},t_{0})-u(y,t_{0}))}{e^{\lambda f(|x_{0}-y^{\alpha}|)}|x_{0}-y^{\alpha}|^{n+sp}},\ \  y\in\Sigma_{\alpha}.
\end{split}
\end{equation}
By the reflection we have
\begin{align}\label{eq2}
&(-\Delta-\lambda_{f})_{p}^{s}u_{\alpha}(x_{0},t_{0})-(-\Delta-\lambda_{f})_{p}^{s} u(x_{0},t_{0})\notag\\
  ={}
  &C_{n,sp}\textit{PV}\int_{\mathbb{R}^{n}}\frac{G(u_{\alpha}(x_{0},t_{0})-u_{\alpha}(y,t_{0}))-G(u(x_{0},t_{0})-u(y,t_{0}))}{e^{\lambda f(|x_{0}-y|)}|x_{0}-y|^{n+sp}}dy\notag\\
  ={}
  &C_{n,sp}\textit{PV}\int_{\Sigma_{\alpha}}\frac{G(u_{\alpha}(x_{0},t_{0})-u_{\alpha}(y,t_{0}))-G(u(x_{0},t_{0})-u(y,t_{0}))}{e^{\lambda f(|x_{0}-y|)}|x_{0}-y|^{n+sp}}dy\notag\\
  +&C_{n,sp}\textit{PV}\int_{\mathbb{R}^{n}\setminus\Sigma_{\alpha}}\frac{G(u_{\alpha}(x_{0},t_{0})-u_{\alpha}(y,t_{0}))-G(u(x_{0},t_{0})-u(y,t_{0}))}{e^{\lambda f(|x_{0}-y|)}|x_{0}-y|^{n+sp}}dy \notag\\
  ={}
  &C_{n,sp}\textit{PV}\int_{\Sigma_{\alpha}}\frac{G(u_{\alpha}(x_{0},t_{0})-u_{\alpha}(y,t_{0}))-G(u(x_{0},t_{0})-u(y,t_{0}))}{e^{\lambda f(|x_{0}-y|)}|x_{0}-y|^{n+sp}}dy\notag\\
  +&C_{n,sp}\textit{PV}\int_{\Sigma_{\alpha}}\frac{G(u_{\alpha}(x_{0},t_{0})-u(y,t_{0}))-G(u(x_{0},t_{0})-u_{\alpha}(y,t_{0}))}{e^{\lambda f(|x_{0}-y^{\alpha}|)}|x_{0}-y^{\alpha}|^{n+sp}}dy\notag\\
   \leq{}
   &C_{n,sp}\textit{PV}\int_{\Sigma_{\alpha}}\frac{G(u_{\alpha}(x_{0},t_{0})-u_{\alpha}(y,t_{0}))-G(u(x_{0},t_{0})-u(y,t_{0}))}{e^{\lambda f(|x_{0}-y^{\alpha}|)}|x_{0}-y^{\alpha}|^{n+sp}}dy\notag\\
  +&C_{n,sp}\textit{PV}\int_{\Sigma_{\alpha}}\frac{G(u_{\alpha}(x_{0},t_{0})-u(y,t_{0}))-G(u(x_{0},t_{0})-u_{\alpha}(y,t_{0}))}{e^{\lambda f(|x_{0}-y^{\alpha}|)}|x_{0}-y^{\alpha}|^{n+sp}}dy\notag\\
  ={}
  &C_{n,sp}\textit{PV}\int_{\Sigma_{\alpha}}\frac{1}{e^{\lambda f(|x_{0}-y^{\alpha}|)}|x_{0}-y^{\alpha}|^{n+sp}}
  \times\big[\{\vphantom{\frac1x}G(u_{\alpha}(x_{0},t_{0})-u_{\alpha}(y,t_{0}))\notag\\
\phantom{=bigg1(x+{}}
  -&G(u(x_{0},t_{0})-u_{\alpha}(y,t_{0}))\}+\{G(u_{\alpha}(x_{0},t_{0})-u(y,t_{0}))\notag\\
  &-G(u(x_{0},t_{0})-u(y,t_{0}))\}\big]dy \notag\\
    ={}
  &C_{n,sp}\textit{PV}\int_{\Sigma_{\alpha}}\frac{G'(\xi(y))(M_{1}(y)-M_{2}(y))+G'(\eta(y))(M_{3}(y)-M_{4}(y))}{e^{\lambda f(|x_{0}-y^{\alpha}|)}|x_{0}-y^{\alpha}|^{n+sp}}dy\notag\\
  ={}
  &C_{n,sp}\textit{PV}\int_{\Sigma_{\alpha}}\frac{w_{\alpha}(x_{0},t_{0})[G'(\xi(y))+G'(\eta(y))]}{e^{\lambda f(|x_{0}-y^{\alpha}|)}|x_{0}-y^{\alpha}|^{n+sp}}dy\notag\\
  ={}
  &C_{n,sp}\frac{\bar{w}_{\alpha}(x_{0},t_{0})}{e^{qt}}\textit{PV}\int_{\Sigma_{\alpha}}\frac{G'(\xi(y))+G'(\eta(y))}{e^{\lambda f(|x_{0}-y^{\alpha}|)}|x_{0}-y^{\alpha}|^{n+sp}}dy,
\end{align}
where $$M_{1}(y)=u_{\alpha}(x_{0},t_{0})-u_{\alpha}(y,t_{0}),\ \ M_{2}(y)=u(x_{0},t_{0})-u_{\alpha}(y,t_{0}),$$
$$M_{3}(y)=u_{\alpha}(x_{0},t_{0})-u(y,t_{0}),\ \ M_{4}(y)=u(x_{0},t_{0})-u(y,t_{0}),$$
$$M_{1}(y)\leq\xi(y)\leq M_{2}(y),\ \ M_{3}(y)\leq\eta(y)\leq M_{4}(y),$$
and the only inequality above is from \eqref{LS1C5}.
Therefore,
\begin{align}\label{eq9}
\begin{split}
&e^{qt}[(-\Delta-\lambda_{f})_{p}^{s}u_{\alpha}(x_{0},t_{0})-(-\Delta-\lambda_{f})_{p}^{s}u(x_{0},t_{0})]
\end{split}\notag\\
\leq{}
&C_{n,sp}\bar{w}_{\alpha}(x_{0},t_{0})\textit{PV}\int_{\Sigma_{\alpha}}\frac{G'(\xi(y))+G'(\eta(y))}{e^{\lambda f(|x_{0}-y^{\alpha}|)}|x_{0}-y^{\alpha}|^{n+sp}}dy\notag\\
<{}
&0.
\end{align}
Combining \eqref{eq8} and \eqref{eq9} yelids
$$(c(x,t)+q)\bar{w}_{\alpha}(x_{0},t_{0})<0.$$

According to the hypothesis condition of $c(x,t)$ in lemma \ref{lem2}, we can assume $0<q<|c(x_{0},t_{0})|$, we have
$$(c(x_{0},t_{0})+q)\bar{w}_{\alpha}(x_{0},t_{0})>0.$$
This is a contradiction.
This contradicts \eqref{eq51}, and \eqref{eq11} holds. That is
\begin{align}\label{eq12}
&\bar{w}_{\alpha}(x,t)\geq\min\{0,\underset{\Omega}\inf\bar{w}_{\alpha}(x,\underline{t})\}\notag \\
={}
&\min\{0, e^{q\underline{t}}\underset{\Omega}\inf w_{\alpha}(x,\underline{t})\}\notag \\
\geq{}
&-C,\ \ \ (x,t)\in\Omega\times[\underline{t},T].
\end{align}
Furthermore, we obtain
$$w_{\alpha}(x,t)\geq-Ce^{-q\underline{t}},\ \forall t>\underline{t}.$$
Let $\underline{t}\rightarrow+\infty$ and we have
\begin{equation}
w_{\alpha}(x,t)\geq0,\ (x,t)\in\Omega\times[0,+\infty).
\end{equation}
Combing \eqref{eq7} and \eqref{eq12}, we know
$$w_{\alpha}(x,t)\geq0,\ (x,t)\in\Sigma_{\alpha}\times[0,+\infty).$$
Furthermore, if $w_{\alpha}(x,t)=0$ at some point $(x_{\ast},t_{\ast})$ in $\Omega$, then
$$w_{\alpha}(x_{\ast},t_{\ast})=\underset{\Sigma_{\alpha}\times\mathbb{R}}\inf\bar{w}_{\alpha}(x,t)=0,\ \ \partial_t w(x_{\ast},t_{\ast})=0,$$
\begin{equation}
(-\Delta-\lambda_{f})_{p}^{s}u_{\alpha}(x_{\ast},t_{\ast})-(-\Delta-\lambda_{f})_{p}^{s}u(x_{\ast},t_{\ast})
<0.
\end{equation}
However, it follows from \eqref{eq7} that
\begin{align*}
0&>\partial_tw_{\alpha}(x_{\ast},t_{\ast})+(-\Delta-\lambda_{f})_{p}^{s}u_{\alpha}(x_{\ast},t_{\ast})-(-\Delta-\lambda_{f})_{p}^{s}u(x_{\ast},t_{\ast})\notag \\
&=c(x_{\ast},t_{\ast})w_{\alpha}(x_{\ast},t_{\ast})\notag \\
&=0.
\end{align*}
This is a contradiction. Therefore, the strong maximum principle also holds. The proof is completed.
\end{proof}

\begin{rem}
If we replace the first equation in \eqref{eq7} by
$$\partial_tw_{\alpha}(x,t)+(-\Delta-\lambda_{f})_{p}^{s}u_{\alpha}(x,t)-(-\Delta-\lambda_{f})_{p}^{s}u(x,t)
\geq0,\ (x,t)\in\Omega\times[0,+\infty),$$
Then lemma \ref{lem2} still holds.
\end{rem}

\begin{lem}\label{narrow}
(Narrow region principle) Let $\Omega\subset\Sigma_{\alpha}$ be a bounded or unbounded narrow region, such that it is contained in $\{x\ |\alpha-\delta<x_{1}<\alpha\}$ with small $\delta$. Assume that $w_{\alpha}\in L_{sp}\cap C_{loc}^{1,1}(\Omega)$ is lower semi-continuous on $\bar{\Omega}$ and uniformly bounded about $t$,
\begin{equation}\label{eq1}
\left\{
\begin{array}{ll}
\partial_tw_{\alpha}(x,t)+(-\Delta-\lambda_{f})_{p}^{s}u_{\alpha}(x,t)-(-\Delta-\lambda_{f})_{p}^{s}u(x,t)
\\ \quad=c(x,t)w_{\alpha}(x,t),
&(x,t)\in\Omega\times[0,+\infty),\\
w_{\alpha}(x,t)\geq0,
&(x,t)\in(\Sigma_{\alpha}\setminus\Omega)\times[0,+\infty),\\
w_{\alpha}(x^{\alpha},t)=-w(x,t),&x\in\Sigma_{\alpha}\times[0,+\infty),
\end{array}
\right.
\end{equation}
where $c(x,t)<0$ is bounded from above in $\Omega$, then
\begin{equation}
\varliminf_{t\rightarrow\infty}w_{\alpha}(x,t)\geq0,\ (x,t)\in\Sigma_{\alpha}\times[0,+\infty),
\end{equation}
for $\delta$ being small enough. Furthermore, if $w_{\alpha}(x,t)=0$ at some point in $\Omega$, then
\begin{equation}
w_{\alpha}(x,t)\equiv0,\ \mbox{almost everywhere in}\ \Omega.
\end{equation}
\end{lem}
\begin{proof}
Let $x=(x_{1},x')$, denote $$\bar{w}_{\alpha}(x,t)=e^{qt}w_{\alpha}(x,t),$$
where $q>0$ is a constant.

According to the definition of $\bar{w}_{\alpha}(x,t)$ and \eqref{eq1}, $\bar{w}(x,t)$ satisfies
\begin{equation}
\left\{
\begin{array}{ll}
\partial_t\bar{w}_{\alpha}(x,t)+e^{qt}[(-\Delta-\lambda_{f})_{p}^{s}u_{\alpha}(x,t)-(-\Delta-\lambda_{f})_{p}^{s}u(x,t)]
\\ \quad=(c(x,t)+q)\bar{w}_{\alpha}(x,t),
&(x,t)\in\Omega\times[0,+\infty),\\
\bar{w}_{\alpha}(x,t)\geq0,
&(x,t)\in(\Sigma_{\alpha}\setminus\Omega)\times[0,+\infty).
\end{array}
\right.
\end{equation}

Now we show
\begin{equation}
\bar{w}_{\alpha}(x,t)\geq\min\{0,\underset{\Omega}\inf\bar{w}_{\alpha}(x,\underline{t})\},\ (x,t)\in\Omega\times[\underline{t},T]
\end{equation}
for any $[\underline{t},T]\subset[0,+\infty)$. Otherwise, there exists $(x_{0},t_{0})\in\Omega\times(\underline{t},T]$ such that
$$\bar{w}_{\alpha}(x_{0},t_{0})=\underset{\Sigma_{\alpha\times(\underline{t},T]}}\inf\bar{w}_{\alpha}(x,t)<\min\{0,\underset{\Omega}\inf\bar{w}_{\alpha}(x,\underline{t})\},$$
then
\begin{equation}
\partial_t\bar{w}_{\alpha}(x_{0},t_{0})\leq0,
\end{equation}
and
\begin{align}\begin{split}
&(-\Delta-\lambda_{f})_{p}^{s}u_{\alpha}(x_{0},t_{0})-(-\Delta-\lambda_{f})_{p}^{s} u(x_{0},t_{0})\\
\leq{}
  &C_{n,sp}\frac{\bar{w}_{\alpha}(x_{0},t_{0})}{e^{qt}}\textit{PV}\int_{\Sigma_{\alpha}}\frac{G'(\xi(y))+G'(\eta(y))}{e^{\lambda f(|x_{0}-y^{\alpha}|)}|x_{0}-y^{\alpha}|^{n+sp}}dy,
\end{split}
\end{align}
where $|x_{0}-y|\leq|x_{0}-y^{\alpha}|$ in $\Sigma_{\alpha}$ and we set
\[H=\{y=(y_{1},y')\in\Sigma_{\alpha}|\delta<y_{1}-(x_{0})_{1}<2\delta,\ |y'-(x_{0})'<\delta|\},\]
then
\begin{align}
\int_{\Sigma_{\alpha}}\frac{G'(\xi(y))+G'(\eta(y))}{e^{\lambda f(|x_{0}-y^{\alpha}|)}|x_{0}-y^{\alpha}|^{n+sp}}dy
\geq\frac{C}{\delta^{sp}},
\end{align}
\begin{align}
e^{qt}[(-\Delta-\lambda_{f})_{p}^{s}u_{\alpha}(x_{0},t_{0})-(-\Delta-\lambda_{f})_{p}^{s} u(x_{0},t_{0})]\leq\frac{C}{\delta^{sp}}\bar{w}_{\alpha}(x_{0},t_{0}).
\end{align}
As a result, we can obtain
\begin{align}\begin{split}\label{eq52}
0\geq{}&
\partial_t\bar{w}_{\alpha}(x_{0},t_{0})\\
={}
&-e^{qt}[(-\Delta-\lambda_{f})_{p}^{s}u_{\alpha}(x_{0},t_{0})-(-\Delta-\lambda_{f})_{p}^{s} u(x_{0},t_{0})]+(c_{\alpha}(x_{0},t_{0})+q)\bar{w}_{\alpha}(x_{0},t_{0})\\
\geq{}
&(-\frac{C}{\delta^{sp}}+c_{\alpha}(x_{0},t_{0})+q)\bar{w}_{\alpha}(x_{0},t_{0}).
\end{split}
\end{align}
Since $c_{\alpha}(x,t)$ is bounded from above for all $(x,t)$, we can choose small $\delta$ and $m$ in some proper sence such that
\[-\frac{C}{\delta^{sp}}+c_{\alpha}(x_{0},t_{0})+q<0,\]
which contradicts \eqref{eq52}.
Thus
\begin{align}
&\bar{w}_{\alpha}(x,t)\geq\min\{0,\underset{\Omega}\inf\bar{w}_{\alpha}(x,\underline{t})\}\notag \\
={}
&\min\{0, e^{q\underline{t}}\underset{\Omega}\inf w_{\alpha}(x,\underline{t})\}\notag \\
\geq{}
&-C.
\end{align}
Furthermore, we obtain
$$w_{\alpha}(x,t)\geq-Ce^{-q\underline{t}},\ \forall t>\underline{t}.$$
Let $\underline{t}\rightarrow+\infty$ and we have
\begin{equation}
\varliminf_{t\rightarrow\infty}w_{\alpha}(x,t)\geq0,\ x\in\Omega.
\end{equation}
\end{proof}

\begin{lem}
(Asymptotic strong maximum principle)

Assume that $u(x,t)\in(C_{loc}^{1,1}(B_{1}(0))\cap\mathcal{L}_{sp})\times C^{1}[0,\infty)$ is a positive solution to
\begin{equation}\label{eq13}
\partial_tu(x,t)+(-\Delta-\lambda_{f})_{p}^{s}u(x,t)=g(t,u(x,t)),
\end{equation}
where $0<s<1$, $p>2,\ n\geq2$.
%and $\Omega$ is a bounded domain in $B_{1}(0)$ with smooth boundary.

Assume that
\begin{equation}\label{eq22}
g(t,0)=0,\ g\ \mbox{is Lispchitz continuous in u uniformly for t}.
\end{equation}
There is some $\varphi\in w(u)$ which is positive somewhere in $B_{1}(0)$,
%If $u$ is nonnegative and $u>0$ somewhere in $\tilde{\Sigma}_{\alpha}$.
then $$\varphi(x)>0\ \mbox{everywhere in}\ B_{1}(0)\ \mbox{for all}\ \varphi\in w(u).$$

In other words, the following strong maximum principle holds for the whole family of functions $\{\varphi|\varphi\in w(u)\}$ simultaneously:

Either
$$\varphi(x)\equiv0\ everywhere\ in\ B_{1}(0)\ for\ all\ \varphi\in w(u),$$
or
$$\varphi(x)>0\ everywhere\ in\ B_{1}(0)\ for\ all\ \varphi\in w(u).$$
\end{lem}
\begin{proof}
For any $\varphi(x)\in w(u)$, there exists $t_{k}$ such that $u(x,t_{k})\rightarrow \varphi(x)$ as $t_{k}\rightarrow\infty$. Set
$$u_{k}(x,t)=u(x,t+t_{k}-1),$$
$$g_{k}(t,u)=g(t+t_{k}-1,u).$$

Then $u_{k}(x,1)\rightarrow\varphi(x)$ as $k\rightarrow\infty$. From regularity theory for parabolic equations, we conclude that there exist some functions $u_{\infty}$ and $\tilde{g}$ such that $u_{k}\rightarrow u_{\infty},\ g_{k}\rightarrow \tilde{g}$ and $u_{\infty}(x,t)$ satisfies
\begin{equation}
\partial_tu_{\infty}+(-\Delta-\lambda_{f})_{p}^{s}u_{\infty}
=\tilde{g}_{u}u_{\infty},\ (x,t)\in B_{1}(0)\times[0,2].
\end{equation}

Since there is some $\varphi\in w(u)$ satisfies $\varphi>0$ somewhere in $B_{1}(0)$, by the continuity, there exists a set $D\subset\subset B_{1}(0)$ such that
\begin{equation}\label{eq23}
\varphi(x)\geq c\geq0,\ x\in D,
\end{equation}
where $c$ is a positive small constant. This means that
$$u_{\infty}(x,1)\geq c>0,\ x\in D.$$

From the continuity of $u_{\infty}(x,t)$, there exists $0<\varepsilon_{0}<1$, such that $$u_{\infty}(x,t)\geq\frac{c}{2}>0,\ (x,t)\in D\times[1-\varepsilon_{0},1+\varepsilon_{0}].$$
Define $\Phi(x)=(1-|x|^{2})_{+}^{s}$, where
\begin{equation*}
\Phi(x)=\begin{cases}
(1-|x|^{2})^{s},&|x|<1,\\
0,&|x|\geq1,
\end{cases}
\end{equation*}
and $\Phi(x)\in C_{0}^{\infty}(B_{\varepsilon}(\bar{x}))$, $\bar{x}\in B_{1}(0)\setminus D$.

There exists $\varepsilon>0$ such that $B_{\varepsilon}(\bar{x})\subset B_{1}(0)\setminus D$. For any $t\in\mathbb{R}$, we structure a subsolution
$$\underline{u}=\chi_{D}(x)u_{\infty}(x,t)+\delta\eta(t)\Phi(x),$$
where $\delta$ is a positive constant,
\begin{equation*}
\chi_{D}(x)= \begin{cases}
1,&x\in D,\\
0,&x\notin D,
\end{cases}
\end{equation*}
and $\eta(t)\in C_{0}^{\infty}([1-\varepsilon_{0},1+\varepsilon_{0}])$ satisfies
\begin{equation*}
\eta(t)= \begin{cases}
1,&t\in[1-\frac{\varepsilon_{0}}{2},1+\frac{\varepsilon_{0}}{2}],\\
0,&x\notin(1-\varepsilon_{0},1+\varepsilon_{0}).
\end{cases}
\end{equation*}

Since $\Phi(x)\in C_{0}^{\infty}(B_{\varepsilon}(\bar{x}))$, $x\in B_{\varepsilon}(\bar{x})$.
%Due to (see \cite{li2021zhangsub})
%\begin{equation}\label{eq19}
%(-\Delta)_{p}^{s}(1-|x|^{2})_{+}^{s}\leq C,\ x\in B_{1}(0).
%\end{equation}
And then we will use the following corrllary.
\begin{corr}\label{corr1}
Let $0<s<1$, $p>2,\ n\geq2$, $\Phi(x)=(1-|x|^{2})_{+}^{s}$, then $(-\Delta-\lambda_{f})_{p}^{s}\Phi(x)$ is uniformly bounded in the unit ball $B_{1}(0)\subset\mathbb{R}^{n}$.
\end{corr}
\begin{proof}
For some fixed $\delta\in(0,1)$, when $|x|<1-\delta$, by \cite{chen2018maximum}, It is straightforward to check that
\begin{align*}
&|(-\Delta-\lambda_{f})_{p}^{s}\Phi(x)|\notag\\
\leq&C\int_{\mathbb{R}^{n}\setminus B_{\delta}(x)}\frac{|\Phi(x)-\Phi(y)|^{p-2}(\Phi(x)-\Phi(y))}{|x-y|^{n+sp}}dy\notag\\
&+C|\nabla \Phi(x)|^{p-2}\int_{B_{\delta}(x)}\frac{|x-y|^{p}}{e^{\lambda f(|x-y|)}|x-y|^{n+sp}}dy\notag\\
\leq&C.
\end{align*}
That is, $(-\Delta-\lambda_{f})_{p}^{s}\Phi(x)$ is bounded for $|x|<1-\delta$. Therefore in the following we only need to consider the case when $|x|$ is close to 1, since $x$ approaching $-1$ is a similar calculation.

Firstly we give a general estimate for $(-\Delta-\lambda_{f})_{p}^{s}\Phi(x)$ when $x$ is close to 1. Without loss of generality, we assume a fixed $x:=(x,0,\cdots,0)\in B_{1}$ is close to $(1,0,\cdots,0)$, and $y:=(y_{1},y_{2},\cdots,y_{n})=(y_{1},\bar{y})$. We omit the constant $C_{n,sp}$ for simplicity. Then
\begin{align*}
&(-\Delta-\lambda_{f})_{p}^{s}\Phi(x)\\
=&\underset{\epsilon\rightarrow0}\lim\int_{\mathbb{R}^{n}\setminus B_{\delta}(x)}\frac{|\Phi(x)-\Phi(y)|^{p-2}(\Phi(x)-\Phi(y))}{e^{\lambda f(|x-y|)}|x-y|^{n+sp}}dy\\
=&\underset{\epsilon\rightarrow0}\lim\int_{y\in\mathbb{R}^{n}:|x-y|\geq\epsilon}
\frac{|(1-x^{2})^{s}-(1-y^{2})_{+}^{s}|^{p-2}[(1-x^{2})^{s}-(1-y^{2})_{+}^{s}]}{e^{\lambda f(|x-y|)}|x-y|^{n+sp}}dy\\
=&\underset{\epsilon\rightarrow0}\lim\int_{y\in\mathbb{R}^{n}:|x-y|\geq\epsilon}\frac{|(1-x^{2})^{s}-(1-y_{1}^{2}-|\bar{y}|^{2})_{+}^{s}|^{p-2}[(1-x^{2})^{s}-(1-y_{1}^{2}-|\bar{y}|^{2})_{+}^{s}]}{e^{\lambda f(|(x-y_{1})^{2}+|\bar{y}|^{2}|^{\frac{1}{2}})}|(x-y_{1})^{2}+|\bar{y}|^{2}|^{\frac{n+sp}{2}}}dy\\
=&\underset{\epsilon\rightarrow0}\lim\int_{y\in\mathbb{R}^{n}:|x-y|\geq\epsilon}\frac{G[(1-x^{2})^{s}-(1-y_{1}^{2}-|\bar{y}|^{2})_{+}^{s}]}{e^{\lambda f(|(x-y_{1})^{2}+|\bar{y}|^{2}|^{\frac{1}{2}})}|(x-y_{1})^{2}+|\bar{y}|^{2}|^{\frac{n+sp}{2}}}dy,
\end{align*}
where $G[t]=|t|^{p-2}t$. Set $z=(z_{1},\bar{z})$, where $z_{1}=x-y_{1},\ \bar{z}=\bar{y}$. Then
\begin{align*}
&(-\Delta-\lambda_{f})_{p}^{s}\Phi(x)\\
=&\underset{\epsilon\rightarrow0}\lim\int_{z\in\mathbb{R}^{n}:|z|\geq\epsilon}\frac{G[(1-x^{2})^{s}-(1-(x-z_{1})^{2}-|\bar{z}|^{2})_{+}^{s}]}{e^{\lambda f(|(z_{1})^{2}+|\bar{z}|^{2}|^{\frac{1}{2}})}|(z_{1})^{2}+|\bar{z}|^{2}|^{\frac{n+sp}{2}}}dz\\
=&(1-x^{2})^{s(p-1)}\underset{\epsilon\rightarrow0}\lim\int_{z\in\mathbb{R}^{n}:|z|\geq\epsilon}
\frac{G[1-(1+\frac{2xz_{1}}{1-x^{2}}-\frac{|z|^{2}}{1-x^{2}})_{+}^{s}]}{e^{\lambda f(|z|)}|z|^{n+sp}}dz.
\end{align*}
Let $w=\frac{2xz}{1-x^{2}}$, where $w_{1}=\frac{2xz_{1}}{1-x^{2}}$ and $\bar{w}\frac{2x\bar{z}}{1-x^{2}}$, then
\begin{align*}
&(-\Delta-\lambda_{f})_{p}^{s}\Phi(x)\\
=&\frac{(2x)^{sp}}{(1-x^{2})}\underset{\epsilon\rightarrow0}\lim\int_{w\in\mathbb{R}^{n}:|w|\geq\frac{2x\epsilon}{1-x^{2}}}
\frac{G[1-(1+w_{1}-\frac{1-x^{2}}{4x^{2}}|w|^{2})_{+}^{s}]}{e^{\lambda f(|\frac{(1-x^{2})w}{2x}|)}|w|^{n+sp}}dw.
\end{align*}
Now we roughly analyze the integration term.
\begin{align*}
&\int_{|w|\geq\frac{2x\epsilon}{1-x^{2}}}
\frac{G[1-(1+w_{1}-\frac{1-x^{2}}{4x^{2}}|w|^{2})_{+}^{s}]}{e^{\lambda f(|\frac{(1-x^{2})w}{2x}|)}|w|^{n+sp}}dw\\
=&\int_{-\infty}^{+\infty}dw_{1}\int_{|w|^{2}\geq\frac{4x^{2}\epsilon^{2}}{(1-x^{2})^{2}}-w_{1}^{2}}
\frac{G[1-(1+w_{1}-\frac{1-x^{2}}{4x^{2}}w_{1}^{2}-\frac{1-x^{2}}{4x^{2}}|\bar{w}|^{2})_{+}^{s}]}{e^{\lambda f(|\frac{(1-x^{2})^{2}}{4x^{2}}w_{1}^{2}+\frac{(1-x^{2})^{2}}{4x^{2}}|\bar{w}|^{2}|^{\frac{1}{2}})}|w_{1}^{2}+|\bar{w}|^{2}|^{\frac{n+sp}{2}}}dw\\
=&C\int_{-\infty}^{+\infty}dw_{1}\int_{\rho^{2}\geq\frac{4x^{2}\epsilon^{2}}{(1-x^{2})^{2}}-w_{1}^{2}}
\frac{G[1-(1+w_{1}-\frac{1-x^{2}}{4x^{2}}w_{1}^{2}-\frac{1-x^{2}}{4x^{2}}\rho^{2})_{+}^{s}]}{e^{\lambda f(|\frac{(1-x^{2})^{2}}{4x^{2}}w_{1}^{2}+\frac{(1-x^{2})^{2}}{4x^{2}}\rho^{2}|^{\frac{1}{2}})}|w_{1}^{2}+\rho^{2}|^{\frac{n+sp}{2}}}\rho^{n-2}d\rho\\
=&C\int_{-\frac{2x\epsilon}{1-x^{2}}}^{\frac{2x\epsilon}{1-x^{2}}}dw_{1}\int_{\sqrt{\frac{4x^{2}\epsilon^{2}}{(1-x^{2})^{2}}-w_{1}^{2}}}^{\infty}
\frac{G[1-(1+w_{1}-\frac{1-x^{2}}{4x^{2}}w_{1}^{2}-\frac{1-x^{2}}{4x^{2}}\rho^{2})_{+}^{s}]}{e^{\lambda f(|\frac{(1-x^{2})^{2}}{4x^{2}}w_{1}^{2}+\frac{(1-x^{2})^{2}}{4x^{2}}\rho^{2}|^{\frac{1}{2}})}|w_{1}^{2}+\rho^{2}|^{\frac{n+sp}{2}}}\rho^{n-2}d\rho\\
&+C\int_{|w_{1}|\geq\frac{2x\epsilon}{1-x^{2}}}dw_{1}\int_{0}^{\infty}
\frac{G[1-(1+w_{1}-\frac{1-x^{2}}{4x^{2}}w_{1}^{2}-\frac{1-x^{2}}{4x^{2}}\rho^{2})_{+}^{s}]}{e^{\lambda f(|\frac{(1-x^{2})^{2}}{4x^{2}}w_{1}^{2}+\frac{(1-x^{2})^{2}}{4x^{2}}\rho^{2}|^{\frac{1}{2}})}|w_{1}^{2}+\rho^{2}|^{\frac{n+sp}{2}}}\rho^{n-2}d\rho\\
:=&CJ_{1}+CJ_{2}.
\end{align*}
When $x\rightarrow1$, $w_{1}\in[-\frac{2x\epsilon}{1-x^{2}},\frac{2x\epsilon}{1-x^{2}}],\ \rho^{2}\geq\sqrt{\frac{4x^{2}\epsilon^{2}}{(1-x^{2})^{2}}-w_{1}^{2}}$,
$$\frac{1}{e^{\lambda f(|\frac{(1-x^{2})^{2}}{4x^{2}}w_{1}^{2}+\frac{(1-x^{2})^{2}}{4x^{2}}\rho^{2}|^{\frac{1}{2}})}}\leq C.$$
By \cite{li2021zhangsub}, we already know that
\begin{align*} \int_{-\frac{2x\epsilon}{1-x^{2}}}^{\frac{2x\epsilon}{1-x^{2}}}dw_{1}\int_{\sqrt{\frac{4x^{2}\epsilon^{2}}{(1-x^{2})^{2}}-w_{1}^{2}}}^{\infty}
\frac{G[1-(1+w_{1}-\frac{1-x^{2}}{4x^{2}}w_{1}^{2}-\frac{1-x^{2}}{4x^{2}}\rho^{2})_{+}^{s}]}{|w_{1}^{2}+\rho^{2}|^{\frac{n+sp}{2}}}\rho^{n-2}d\rho
&\leq C.
\end{align*}
So $J_{1}<C$. Similarly, we can get $J_{2}<C$. Hence we have completed the proof.
\end{proof}
According to the definition of the operator $(-\Delta-\lambda_{f})_{p}^{s}$ and the assumption, for each fixed $t\in[1-\varepsilon_{0},1+\varepsilon_{0}]$ and for any $x\in B_{\varepsilon}(\bar{x})$, it holds
\begin{align}\label{eq17}
&(-\Delta-\lambda_{f})_{p}^{s}(\chi_{D}(x)u_{\infty}(x,t))\notag\\
 ={}
 &C_{n,sp}\textit{PV}\int_{\mathbb{R}^{n}}\frac{G(\chi_{D}(x)u_{\infty}(x,t)-\chi_{D}(y)u_{\infty}(y,t))}{e^{\lambda f(|x-y|)}|x-y|^{n+sp}}dy\notag\\
 ={}
 &C_{n,sp}\textit{PV}\int_{\mathbb{R}^{n}}\frac{G(-\chi_{D}(y)u_{\infty}(y,t))}{e^{\lambda f(|x-y|)}|x-y|^{n+sp}}dy\notag\\
 ={}
 &C_{n,sp}\textit{PV}\int_{D}\frac{G(-u_{\infty}(y,t))}{e^{\lambda f(|x-y|)}|x-y|^{n+sp}}dy\notag\\
 \leq{}
 &-C,
\end{align}
for each fixed $t\in[1-\varepsilon_{0},1+\varepsilon_{0}]$ and for any $x\in B_{\varepsilon}(\bar{x})$.

For any $(x,t)\in B_{\varepsilon}(\bar{x})\times[1-\varepsilon_{0},1+\varepsilon_{0}]$, it follows from corollary \ref{corr1} that
\begin{align}
&(-\Delta-\lambda_{f})_{p}^{s}\underline{u}(x,t)\notag\\
=&(-\Delta-\lambda_{f})_{p}^{s}(\chi_{D}(x)u_{\infty}(x,t)+\delta\eta(t)\Phi(x))\notag\\
=&(-\Delta-\lambda_{f})_{p}^{s}(\chi_{D}(x)u_{\infty}(x,t))+\delta\eta(t)(-\Delta-\lambda_{f})_{p}^{s}\Phi(x)\notag\\
\leq&-C+C\delta\eta(t).
\end{align}
By a simple calculation, we obtain that for any $(x,t)\in B_{\varepsilon}(\bar{x})\times[1-\varepsilon_{0},1+\varepsilon_{0}]$
\begin{align}\label{eq21}
&\partial_t\underline{u}(x,t)+(-\Delta-\lambda_{f})_{p}^{s}\underline{u}(x,t)\notag\\
=&\delta\eta'(t)\Phi(x)+(-\Delta-\lambda_{f})_{p}^{s}[\chi_{D}(x)u_{\infty}(x,t)+\delta\eta(t)\Phi(x)]\notag\\
\leq&-C+\delta[C\eta(t)+\eta'(t)\Phi(x)].
\end{align}
For sufficiently small $\delta$, it holds
\begin{equation}
\partial_t\underline{u}(x,t)+(-\Delta-\lambda_{f})_{p}^{s}\underline{u}(x,t)\leq0,\ (x,t)\in B_{\varepsilon}(\bar{x})\times[1-\varepsilon_{0},1+\varepsilon_{0}].
\end{equation}

In all cases, for $(x,t)\in B_{\varepsilon}^{c}(\bar{x})\times[1-\varepsilon_{0},1+\varepsilon_{0}]$, we have
$$u_{\infty}(x,t)\geq\chi_{D}(x)u_{\infty}(x,t)=\underline{u}(x,t).$$

Denote
$$w(x,t)=u_{\infty}(x,t)-\underline{u}(x,t).$$
Combining \eqref{eq13} and \eqref{eq21}, for $(x,t)\in B_{\varepsilon}(\bar{x})\times[1-\varepsilon_{0},1+\varepsilon_{0}]$, we have
\begin{align}
&\partial_t w(x,t)+(-\Delta-\lambda_{f})_{p}^{s}u_{\infty}(x,t)-(-\Delta-\lambda_{f})_{p}^{s}\underline{u}(x,t)\notag\\
\geq&g(t_{\infty},u_{\infty})+C_{2}-\delta[C\eta(t)+\eta'(t)\Phi(x)]\notag\\
=&C_{3}(x,t)u_{\infty}(x,t)+C_{2}-\delta[C\eta(t)+\eta'(t)\Phi(x)],
\end{align}
where $p\geq2,\ C_{3}(x,t)=\frac{g(t,u_{\infty})-g(t,0)}{u_{\infty}(x,t)-0}=\frac{g(t,u_{\infty})}{u_{\infty}(x,t)}$ is bounded by \eqref{eq22}. Taking $\delta$ sufficiently small, we obtain
$$\partial_t w(x,t)+(-\Delta-\lambda_{f})_{p}^{s}u_{\infty}(x,t)-(-\Delta-\lambda_{f})_{p}^{s}\underline{u}(x,t)\geq0,$$
$$(x,t)\in B_{\varepsilon}(\bar{x})\times[1-\varepsilon_{0},1+\varepsilon_{0}].$$
Hence
\begin{equation}
\left\{
\begin{array}{ll}
\partial_t w(x,t)+(-\Delta-\lambda_{f})_{p}^{s}u_{\infty}(x,t)-(-\Delta-\lambda_{f})_{p}^{s}\underline{u}(x,t)
\\ \quad\geq0,
&(x,t)\in B_{\varepsilon}(\bar{x})\times[1-\varepsilon_{0},1+\varepsilon_{0}],\\
w(x,t)\geq0,
&(x,t)\in B_{\varepsilon}^{c}(\bar{x})\times[1-\varepsilon_{0},1+\varepsilon_{0}],\\
w(x,0)\geq0,&x\in B_{\varepsilon}(\bar{x}).
\end{array}
\right.
\end{equation}
Applying lemma \ref{lem2} to $w(x,t)$, it yields
$$w(x,t)\geq0,\ (x,t)\in B_{\varepsilon}(\bar{x})\times[1-\varepsilon_{0},1+\varepsilon_{0}].$$
It follows from the definition of $w(x,t)$ that
$$u_{\infty}(x,t)\geq\delta\eta(t)\Phi(x),\ (x,t)\in B_{\varepsilon}(\bar{x})\times[1-\varepsilon_{0},1+\varepsilon_{0}],$$
and
$$u_{\infty}(x,t)\geq\delta\eta(t)(1-|x|^{2})^{s},\ x\in B_{\varepsilon}(\bar{x}).$$
In particular, taking $t=1$, we obtain
$$u_{\infty}(x,1)\geq\delta (1-|x|^{2})^{s},\ x\in B_{\varepsilon}(\bar{x}),$$
and
\begin{equation}\label{eq24}
\varphi(\bar{x})>0.
\end{equation}
By the arbitrariness of $\bar{x}\in B_{1}(0)\setminus D$, combining \eqref{eq23} and \eqref{eq24} we obtain
$$\varphi(x)>0,\ x\in B_{1}(0).$$
\end{proof}
\section{Proof of the Main Result}
\begin{thm}\label{corr2}
(Asymptotic Hopf's lemma for a generalized tempered fractional p-Laplacian) Assume that $u(x,t)\in(C_{loc}^{1,1}(B_{1}(0))\cap\mathcal{L}_{sp})\times C^{1}(0,\infty)$ is a positive solution to
\begin{equation}
\partial_t u(x,t)+(-\Delta-\lambda_{f})_{p}^{s}u(x,t)=g(t,u(x,t)),
\end{equation}
where $0<s<1$, $p>2,\ n\geq2$.
%and $\Omega$ is a bounded domain in $B_{1}(0)$ with smooth boundary.

Assume that
\begin{equation}
g(t,0)=0,\ g\ \mbox{is Lispchitz continuous in u uniformly fo t}.
\end{equation}
Then there exists a positive constant $c$, such that for any $t\rightarrow\infty$ and for all $x$ near the boundary of $B_{1}(0)$, we have
\[\varphi(x)\geq c d^{s}(x),\]
where $d(x)=dist(x,\partial B_{1}(0))$. It follows that
\[\frac{\partial \varphi}{\partial\nu^{s}}(x_{0})<0,\]
where $\nu$ is the outward normal of $\partial B_{1}(0)$ at $x_{0}$ and $\frac{\partial \varphi}{\partial\nu^{s}}$ is the normal derivative of fractional order $s$.
\end{thm}
\begin{proof}
Through the accurate calculation of the above lemma, we get
$$u_{\infty}(x,t)\geq\delta\eta(t)(1-|x|^{2})^{s},\ x\in B_{\varepsilon}(\bar{x}).$$

In particular, fixed $\delta$, for $x\in B_{\varepsilon}(\bar{x})$, one has
\[u_{\infty}(x,1)\geq\delta((1-|x|)(1+|x|))^{s}=cd^{s}(x),\]
where $d(x)=dist(x,\partial B_{1}(0))=1-|x|,\ c>0.$ Consequently,
\[\underset{x\rightarrow\partial B_{1}(0)}\lim\frac{u_{\infty}(x,1)}{d^{s}(x)}\geq c>0.\]

Hence if $\nu$ is the outward normal of $\partial\Omega$ at $x^{0}\in\partial B_{1}(0)$, then
\[\frac{\partial u_{\infty}(x^{0},1)}{\partial\nu^{s}}<0,\]
due to $u(x^{0},1)=0,\ x^{0}\in\partial B_{1}(0)$. It follows that
\[\frac{\partial\varphi}{\partial\nu^{s}}(x^{0})<0,\ \forall x^{0}\in\partial B_{1}(0).\]
\end{proof}

Finally, we briefly explain how to apply Hopf's lemma for parabolic equations involving a generalized tempered fractional $p$-Laplacian in the first step of the method of moving planes, take $\Omega=B_{1}(0)$ as an example. To obtain the radial symmetry of positive solutions to
\begin{equation}
\left\{
\begin{array}{ll}
\partial_tu(x,t)+(-\Delta-\lambda_{f})_{p}^{s}u(x,t)=g(t,u(x,t)),
&(x,t)\in B_{1}(0)\times[0,+\infty),\\
u(x)=0,&(x,t)\in B_{1}^{c}(0)\times[0,+\infty),
\end{array}
\right.
\end{equation}
where $0<s<1$, $p>2,\ n\geq2$.

Let $\Omega_{\alpha}=\{x\in\Omega|x_{1}<\alpha\}$. In this step, we show that, for $\alpha>-1$ and sufficiently closed to $-1$, it holds
\begin{equation}\label{eq28}
\psi_{\alpha}\geq0,\ \forall x\in\Omega_{\alpha},\ \mbox{for all}\ \varphi(x)\in w(u).
\end{equation}
By applying the theorem \ref{corr2}, we have
\begin{equation}
\frac{\partial\varphi}{\partial x_{1}^{s}}(x^{0})>0,\ \forall x^{0}\in\partial B_{1}(0).
\end{equation}
As a consequence,
\[\frac{\partial\varphi}{\partial x_{1}}(x^{0})>0,\ \forall x^{0}\in\partial B_{1}(0).\]
Then by the continuity of $\frac{\partial\varphi}{\partial x_{1}^{s}}$ in some proper sence, it is natural to expect that
\begin{equation}
\frac{\partial\varphi}{\partial x_{1}}(x)>0,\ \mbox{for}\ x\in\Omega\ \mbox{sufficiently close to $x^{0}$,}
\end{equation}
which implies \eqref{eq28} immediately.
%\section{Data Availability Statement}

%No data, models, or code were generated or used during the study.

%\section{Conflicts of Interest}
%Authors state no conflict of interest.
%Plain
\bibliographystyle{unsrt}

\end{sloppypar}
\end{document}